\documentclass[11pt]{article}
\usepackage{amsmath,amssymb,amsthm,mathtools}
\usepackage[margin=1in]{geometry}
\usepackage{enumitem}
\usepackage{microtype}
\usepackage[hidelinks]{hyperref}

\newtheorem{statement}{Statement}[section]
\newtheorem{theorem}[statement]{Theorem}
\newtheorem{lemma}[statement]{Lemma}
\newtheorem{proposition}[statement]{Proposition}

\newtheorem{remark}[statement]{Remark}
\newtheorem{definition}[statement]{Definition}
\newtheorem{example}[statement]{Example}

\newcommand{\R}{\mathbb{R}}
\newcommand{\Pp}{\mathbb{P}}
\newcommand{\Ee}{\mathbb{E}}
\newcommand{\1}{\mathbf 1}
\newcommand{\norm}[1]{\left\lVert #1\right\rVert}
\newcommand{\ip}[2]{\left\langle #1,#2\right\rangle}

\numberwithin{equation}{section}

\allowdisplaybreaks[4]
\begin{document}

\title{Non-blow-up of stochastic heat equations  by noise}
\author{Yunfeng Chen$^{1}$,
		Tusheng Zhang$^{2,3}$}
\footnotetext[1]{\, School of Mathematics, Hefei University of Technology, Hefei, China. Email: yfchen@hfut.edu.cn}
	\footnotetext[2]{\, School of Mathematics, University of Science and Technology of China, Hefei, China.}
	\footnotetext[3]{\, Department of Mathematics, University of Manchester, Manchester M13 9PL, United Kingdom. Email: tusheng.zhang@manchester.ac.uk }
\date{}
\maketitle
\begin{abstract}
Consider the  heat equation:
\begin{equation}\label{eq:00}
\begin{cases}
\displaystyle
 du(t,x)=\frac12\Delta u(t,x)\,dt+b(u(t,x))\,dt,
 &t>0,\ x\in D,\\[1mm]
 u(t,x)=0,&t\ge0,\ x\in\partial D,\\
 u(0,x)=u_0(x),&x\in D.
\end{cases}
\end{equation}
It is well known that superlinear growth of the coefficient $b$ may lead to finite-time blow-up. In this paper, we show that, for every locally Lipschitz reaction coefficient $b:\R\to\R$, there always exists a locally Lipschitz noise coefficient $\sigma:\R\to\R$ such that the stochastic heat equation
\begin{equation}\label{eq:01}
\begin{cases}
\displaystyle
 du(t,x)=\frac12\Delta u(t,x)\,dt+b(u(t,x))\,dt+\sigma(u(t,x))\,dB_t,
 &t>0,\ x\in D,\\[1mm]
 u(t,x)=0,&t\ge0,\ x\in\partial D,\\
 u(0,x)=u_0(x),&x\in D.
\end{cases}
\end{equation}
has a unique global solution for every $u_0\in C_c(D)$. Here $(B_t)_{t\ge0}$ is a real-valued Brownian motion. In particular, if $b=0$, the equation has a global solution for every locally Lipschitz $\sigma$, without any growth restriction at infinity. The proof uses reflected scalar barriers, a Lyapunov estimate, truncation, and comparison techniques.
\end{abstract}

\section{Introduction and main result}

Let $D\subset\R^d$ be a bounded domain with smooth boundary. Consider the heat equation

\begin{equation}\label{eq:pde}
\begin{cases}
\displaystyle
 du(t,x)=\frac12\Delta u(t,x)\,dt+b(u(t,x))\,dt,
 &t>0,\ x\in D,\\[1mm]
 u(t,x)=0,&t\ge0,\ x\in\partial D,\\
 u(0,x)=u_0(x),&x\in D.
\end{cases}
\end{equation}
It is well known that  superlinear outward reaction terms $b$ may lead to finite-time blow-up. A classical example is \(b(z)=|z|^{p-1}z\), \(p>1\), for which sufficiently large nonnegative initial data lead to finite-time blow-up \cite{Kaplan1963,QuittnerSouplet2019}.
In this paper, we investigate how multiplicative noise can prevent finite-time blow-up in the heat equation. More precisely, let
$(\Omega,\mathcal F,(\mathcal F_t)_{t\ge0},\mathbb P)$
be a filtered probability space satisfying the usual conditions,
and let $(B_t)_{t\ge0}$ be a real-valued
$(\mathcal F_t)$-Brownian motion.
We consider the heat equation with random forcing:
\begin{equation}\label{eq:spde}
\begin{cases}
\displaystyle
 du(t,x)=\frac12\Delta u(t,x)\,dt+b(u(t,x))\,dt+\sigma(u(t,x))\,dB_t,
 &t>0,\ x\in D,\\[1mm]
 u(t,x)=0,&t\ge0,\ x\in\partial D,\\
 u(0,x)=u_0(x),&x\in D.
\end{cases}
\end{equation}

We will show that, for every locally Lipschitz reaction coefficient $b$, one can choose a locally Lipschitz noise coefficient $\sigma$ so that the stochastic heat equation \eqref{eq:spde} admits a unique global solution for every $u_0\in C_c(D)$. In particular, if $b=0$, the equation has a global solution for every locally Lipschitz $\sigma$, without any growth restriction at infinity.

It has been proved in the literature for various types of partial differential equations
that adding suitable noise can prevent the blow-up of solutions; see, for example,
\cite{CZZ,HLL,RTW,TY}. In these particular works, the noise coefficients are
typically non-local in space, since they depend on global quantities such as
the Hilbert space norm of the solution; a representative form is
\[
|u(t)|_{H}^{m}\sigma(u(t,x)).
\]
It is therefore natural to ask whether noise with a pointwise nonlinear coefficient can also prevent blow-up.

A related but distinct line of research concerns stochastic heat equations driven
by space--time white noise. In this setting, global existence results are often
obtained under logarithmic growth conditions, such as
\[
|b(z)|\lesssim |z|\log_{+}|z|,
\]
or under additional restrictions on the growth of the multiplicative coefficient;
see, for example, \cite{DalangKhoshnevisanZhang2019}. For polynomially growing
local multiplicative noise, the situation is more delicate: in one-dimensional
models, global existence and finite-time blow-up depend strongly on the growth
exponent, with the critical exponent \(3/2\) appearing in the results of
\cite{Mueller1991,Mueller2000,MuellerSowers1993}. Salins \cite{Salins2022TAMS} established global existence under
an Osgood-type condition on the drift and a related growth bound
on the diffusion coefficient. More recently, Salins
\cite{Salins2026} treated polynomially growing multiplicative
space--time white noise under additional conditions linking
the drift and noise.

The present paper considers a different noise structure: a single real-valued
Brownian motion acts coherently at every spatial point, while the diffusion
coefficient is the pointwise nonlinear function $\sigma(u(t,x))$. We prove that,
for every locally Lipschitz reaction coefficient, a suitable local noise
prevents finite-time blow-up. The proof uses a Lyapunov estimate and
scalar barriers reflected at zero.

Throughout the paper, the scalar coefficients are locally Lipschitz. We first specify the solution class.

\begin{definition}
We say that an adapted process $u$ is a global solution to equation \eqref{eq:spde} if, for every $T<\infty$,
\begin{equation}\label{eq:solution-class}
 u\in C([0,T];L^2(D))\cap L^2(0,T;H_0^1(D))
 \quad\text{a.s.}
\end{equation}
and $u$ has a jointly measurable representative such that
\begin{equation}\label{eq:Linfty-local}
 \sup_{0\le t\le T}\norm{u(t)}_{L^\infty(D)}<\infty
 \quad\text{a.s.,}
\end{equation}
and, for every $\varphi\in H_0^1(D)$, the following identity holds almost surely for all \(t\ge0\),
\begin{align}\label{eq:variational-form}
 \ip{u(t)}{\varphi}_{L^2}
 &=\ip{u_0}{\varphi}_{L^2}
 -\frac12\int_0^t\int_D\nabla u(s,x)\cdot\nabla\varphi(x)\,dx\,ds\notag\\
 &\quad+\int_0^t\ip{b(u(s))}{\varphi}_{L^2}\,ds
 +\int_0^t\ip{\sigma(u(s))}{\varphi}_{L^2}\,dB_s,
\end{align}
with the stochastic integral understood locally.
\end{definition}


 Here is the main result of the paper.

\begin{theorem}[Non-explosion by a suitable noise]
	\label{thm:main}
	For every locally Lipschitz function $b:\R\to\R$, there
	exists a locally Lipschitz function $\sigma:\R\to\R$
	such that, for every $u_0\in C_c(D)$, equation
	\eqref{eq:spde} has a unique global solution in the class
	\eqref{eq:solution-class}--\eqref{eq:variational-form}.
	
	More precisely, this conclusion holds for any locally
	Lipschitz noise coefficient $\sigma$  such that
	\begin{equation}\label{eq:LV-bound}
	\sup_{z\in\R}
	\left\{
	\frac{2zb(z)}{1+z^2}
	+\frac{|\sigma(z)|^2(1-z^2)}{(1+z^2)^2}
	\right\}<\infty.
	\end{equation}
\end{theorem}

\begin{remark}
We remark that  coefficient $\sigma$ satisfying (\ref{eq:LV-bound}) always exists. A sufficient condition for  \eqref{eq:LV-bound} to hold  is to choose $\sigma$ so that $4z b(z)\le|\sigma(z)|^2$.
\end{remark}

Next, we provide several examples of how the noise is chosen.

\begin{example}[A polynomially growing drift]\label{rem:polynomial}
Let $n\ge2$ be an integer, $p\ge1$ a real number, and $C>0$. Take
\[
  b(z)=C|z|^{p-1}z,\qquad \sigma(z)=z^n.
\]
Both coefficients are locally Lipschitz. If
\[
p<2n-1,
\qquad\text{or}\qquad
p=2n-1\ \text{and}\ C<\frac12,
\]
then
\begin{eqnarray*}
&&\frac{2zb(z)}{1+z^2}
	+\frac{|\sigma(z)|^2(1-z^2)}{(1+z^2)^2}\\
&=&\frac{2C|z|^{p+1}}{1+z^2}
+\frac{|z|^{2n}(1-z^2)}{(1+z^2)^2}
\longrightarrow-\infty
\qquad\text{as }|z|\to\infty.
\end{eqnarray*}
Hence \eqref{eq:LV-bound} holds. Thus, by Theorem \ref{thm:main}, the equation
\[
 du(t,x)=\frac12\Delta u(t,x)\,dt
 +C|u(t,x)|^{p-1}u(t,x)\,dt
 +u(t,x)^n\,dB_t,
\]
with the initial and boundary conditions in \eqref{eq:spde} has a unique global solution.
\end{example}

\begin{example}[An exponentially growing drift]\label{ex:exponential}
Take
\[
 b(z)=z e^{z^2},
 \qquad \sigma(z)=2z e^{z^2/2}.
\]
Both coefficients are locally Lipschitz, and $4z b(z)=|\sigma(z)|^2$. Hence, \eqref{eq:LV-bound} holds, and the equation
\[
 du=\frac12\Delta u\,dt+u e^{u^2}\,dt
     +2u e^{u^2/2}\,dB_t
\]
with the initial and boundary conditions in \eqref{eq:spde} has a unique global solution. This example allows drift growth faster than any polynomial.
\end{example}

\begin{example}[A superlinear drift with a logarithmic factor]
	Take
	\[
	b(z)=z[\log(1+z^2)]^2,
	\qquad
	\sigma(z)=\sqrt{3}\,z\log(1+z^2).
	\]
	Both coefficients are locally Lipschitz. The corresponding equation is
	\[
	du=\frac12\Delta u\,dt
	+u[\log(1+u^2)]^2\,dt
	+\sqrt{3}\,u\log(1+u^2)\,dB_t,
	\]
	with the initial and boundary conditions in \eqref{eq:spde}.
	
One can once again verify that  \eqref{eq:LV-bound} holds, and
Theorem \ref{thm:main} yields a unique
	global solution for every $u_0\in C_c(D)$.
\end{example}

\begin{example}[Coefficients that do not vanish at zero]\label{ex:nonzero-origin}
Take
\[
 b(z)=1+z^3,
 \qquad \sigma(z)=1+2z^2.
\]
Here $b(0)=\sigma(0)=1$, and
\[
 \frac14|\sigma(z)|^2-z b(z)
 =\left(z-\frac12\right)^2\ge0,
\]
which implies  \eqref{eq:LV-bound}. Consequently, the equation
\[
 du=\frac12\Delta u\,dt+(1+u^3)\,dt
     +(1+2u^2)\,dB_t
\]
with the initial and boundary conditions in \eqref{eq:spde} has a unique global solution.
\end{example}

\vskip 0.4cm
The rest of the paper is organized as follows. In Section 2, we establish the global well-posedness of the reflected scalar barrier equations. In Section 3, we introduce truncated stochastic heat equations and reflected scalar stochastic differential equations, and compare their solutions. Section 4 is devoted to the proof of the main result.

\section{Reflected SDE barriers}
For $\epsilon\in\{-1,1\}$ and $x\in\R$ with $\epsilon x\ge0$, consider the reflected stochastic differential equation
\begin{equation}\label{eq:scalar}
 dX_t=b(X_t)\,dt+\sigma(X_t)\,dB_t+\epsilon\,dK_t,
 \qquad X_0=x.
\end{equation}
Here $X$ and $K$ are continuous adapted processes, $\epsilon X_t\ge0$, and $K$ is nondecreasing with $K_0=0$ and
\begin{equation}\label{eq:reflection-support}
 \int_0^t\1_{\{X_s\ne0\}}\,dK_s=0,
 \qquad t\ge0.
\end{equation}
Thus $K$ increases only when $X$ is at zero. The choice $\epsilon=1$ gives reflection into $[0,\infty)$, and $\epsilon=-1$ gives reflection into $(-\infty,0]$.

\begin{proposition}\label{lem:scalar}
Suppose that $b$ and $\sigma$ are locally Lipschitz and satisfy \eqref{eq:LV-bound}. For either choice of $\epsilon$ and every $x$ with $\epsilon x\ge0$, equation \eqref{eq:scalar}, subject to \eqref{eq:reflection-support}, has a unique global strong solution $(X,K)$.
\end{proposition}

\begin{proof}
	We first truncate the coefficients outside a bounded interval. This gives bounded, globally Lipschitz coefficients. The existence and uniqueness of solutions to each truncated equation follow from \cite[Theorem 1.2.1]{Pilipenko2014}. For the negative half-line, apply the same result to $-X$. By localization,  there is a unique maximal strong solution up to its
	explosion time. We show that its lifetime is infinite.

For $V(z)=\log(1+z^2)$,
\[
 V'(z)=\frac{2z}{1+z^2},
 \qquad V''(z)=\frac{2(1-z^2)}{(1+z^2)^2}.
\]
Since $V'(0)=0$, the support condition \eqref{eq:reflection-support} gives
\[
 \int_0^t V'(X_s)\,dK_s=0
\]
up to the lifetime of the solution. Thus reflection makes no contribution to the Lyapunov estimate. Define the differential operator
\[
\mathcal L f(z)
:=b(z)f'(z)+\frac12|\sigma(z)|^2f''(z),
\qquad f\in C^2(\R).
\]

For $R>|x|$, define
\[
 \rho_R^x:=\inf\{t\ge0:|X_t|\ge R\}.
\]
It\^o's formula and \eqref{eq:LV-bound} give
\[
 \Ee V(X_{t\wedge\rho_R^x})
 =V(x)+\Ee\int_0^{t\wedge\rho_R^x}\mathcal LV(X_s)\,ds
 \le V(x)+C_Vt.
\]
Here, \(C_V:=\max\{0,\sup_{z\in\mathbb R}\mathcal LV(z)\}<\infty\). On $\{\rho_R^x\le t\}$, continuity gives $|X_{t\wedge\rho_R^x}|=R$, and therefore
\begin{equation}\label{eq:scalar-exit-prob}
 \Pp(\rho_R^x\le t)\log(1+R^2)
 \le\log(1+x^2)+C_Vt.
\end{equation}
Letting $R\to\infty$ shows that the maximal lifetime is infinite almost surely. Moreover, equation \eqref{eq:scalar}, the local boundedness of the coefficients, and the continuity of \(X\) imply that \(K_T<\infty\) almost surely for every \(T<\infty\).
\end{proof}

In particular, on the same probability space and with the same Brownian motion as \eqref{eq:spde}, let $(X^+,K^+)$ and $(X^-,K^-)$ solve
\begin{equation}\label{eq:scalar-barriers}
 dX_t^\pm=b(X_t^\pm)\,dt+\sigma(X_t^\pm)\,dB_t
          \pm dK_t^\pm,
 \qquad X_0^+=M,\quad X_0^-=-M,
\end{equation}
where $M:=\norm{u_0}_{L^\infty(D)}$. These global reflected barriers satisfy
\begin{equation}\label{eq:scalar-sign}
 X_t^-\le0\le X_t^+,
 \qquad t\ge0,
 \quad\text{a.s.}
\end{equation}

\section{Globally Lipschitz truncations}
In this section, we compare solutions of the truncated stochastic heat equations with reflected scalar SDE barriers.
\subsection{Truncated equations}
For $N>M+1$, let
\[
 \pi_N(r):=(-N)\vee(r\wedge N),
\]
and define
\begin{equation}\label{eq:truncated-coefficients}
 b_N(r):=b(\pi_N(r)),
 \qquad
 \sigma_N(r):=\sigma(\pi_N(r)).
\end{equation}
Because $b$ and $\sigma$ are Lipschitz on $[-N,N]$, both $b_N$ and $\sigma_N$ are globally Lipschitz. Moreover,
\begin{equation}\label{eq:trunc-equals}
 b_N(r)=b(r),\qquad \sigma_N(r)=\sigma(r),
 \qquad |r|\le N,
\end{equation}
and both truncated coefficients are bounded.

Consider the truncated SPDE
\begin{equation}\label{eq:truncated-spde}
\begin{cases}
\displaystyle
 du_N(t,x)=\frac12\Delta u_N(t,x)\,dt
 +b_N(u_N(t,x))\,dt
 +\sigma_N(u_N(t,x))\,dB_t,
 &t>0,\ x\in D,\\
 u_N(t,x)=0,&t\ge0,\ x\in\partial D,\\
 u_N(0,x)=u_0(x),&x\in D.
\end{cases}
\end{equation}

\begin{lemma}[Global well-posedness of the truncated equation]\label{lem:truncated-wp}
For every $N>M+1$, equation \eqref{eq:truncated-spde} has a unique global adapted solution satisfying, for each $T<\infty$,
\begin{equation}\label{eq:trunc-reg}
 u_N\in C([0,T];L^2(D))\cap L^2(0,T;H_0^1(D))
 \quad\text{a.s.}
\end{equation}
Moreover, the usual finite-moment energy estimates hold on every bounded time interval.
\end{lemma}

\begin{proof}
Let $H=L^2(D)$ and let $A=\frac12\Delta$ with homogeneous Dirichlet boundary condition. Then $A$ generates an analytic contraction semigroup $S(t)$ on $H$. Define the Nemytskii maps
\[
 F_N(v)(x):=b_N(v(x)),
 \qquad
 G_N(v)(x):=\sigma_N(v(x)).
\]
Since both scalar functions are globally Lipschitz, there are constants $L_{b,N}$ and $L_{\sigma,N}$ such that
\[
 \norm{F_N(v)-F_N(w)}_H\le L_{b,N}\norm{v-w}_H,
 \qquad
 \norm{G_N(v)-G_N(w)}_H\le L_{\sigma,N}\norm{v-w}_H.
\]
Thus the mild equation
\begin{equation}\label{eq:mild-trunc}
 u_N(t)=S(t)u_0
 +\int_0^tS(t-s)F_N(u_N(s))\,ds
 +\int_0^tS(t-s)G_N(u_N(s))\,dB_s
\end{equation}
has a unique global $H$-valued adapted solution by a standard fixed-point argument;  see, for example, \cite{DPZ2014,PrevotRockner2007}.

The variational It\^o formula \cite{PrevotRockner2007} yields the energy identity
\begin{align*}
 \norm{u_N(t)}_2^2
 &+\int_0^t\norm{\nabla u_N(s)}_2^2\,ds\\
 &=\norm{u_0}_2^2
 +2\int_0^t\ip{u_N(s)}{F_N(u_N(s))}_{L^2}\,ds
 +\int_0^t\norm{G_N(u_N(s))}_2^2\,ds\\
 &\quad+2\int_0^t\ip{u_N(s)}{G_N(u_N(s))}_{L^2}\,dB_s.
\end{align*}
Boundedness of the truncated coefficients and boundedness of $D$ give
\[
 |\ip{v}{F_N(v)}|+\norm{G_N(v)}_2^2
 \le C_N(1+\norm{v}_2^2).
\]
The Burkholder–Davis–Gundy inequality and Gronwall's lemma therefore yield the claimed finite-horizon estimates and \eqref{eq:trunc-reg}.
\end{proof}

On the same probability space, let $(Y^{N,\pm},K^{N,\pm})$ solve the reflected equations
\begin{equation}\label{eq:truncated-scalar}
 dY_t^{N,\pm}=b_N(Y_t^{N,\pm})\,dt
 +\sigma_N(Y_t^{N,\pm})\,dB_t\pm dK_t^{N,\pm},
 \qquad Y_0^{N,+}=M,\quad Y_0^{N,-}=-M.
\end{equation}
Here $K^{N,\pm}$ are continuous nondecreasing processes starting at zero, with
\[
 \int_0^t\1_{\{Y_s^{N,\pm}\ne0\}}\,dK_s^{N,\pm}=0.
\]
Since $b_N$ and $\sigma_N$ are bounded and globally Lipschitz, these equations have unique global strong solutions by \cite[Theorem 1.2.1]{Pilipenko2014}. By their reflecting constraints,
\begin{equation}\label{eq:YN-sign}
 Y_t^{N,-}\le0\le Y_t^{N,+},
 \qquad t\ge0,
 \quad\text{a.s.}
\end{equation}

\subsection{An It\^o formula}

Choose a nondecreasing $\vartheta\in C^\infty(\R)$ such that
\begin{equation}\label{eq:theta}
 0\le\vartheta\le1,
 \qquad
 \vartheta(r)=0\ \text{for }r\le0,
 \qquad
 \vartheta(r)=1\ \text{for }r\ge1.
\end{equation}
For $\varepsilon>0$, define $\phi_\varepsilon\in C^2(\R)$ by
\begin{equation}\label{eq:phi-eps}
 \phi_\varepsilon''(r)=\vartheta(r/\varepsilon),
 \qquad
 \phi_\varepsilon'(r)=\int_0^r\phi_\varepsilon''(q)\,dq,
 \qquad
 \phi_\varepsilon(r)=\int_0^r\phi_\varepsilon'(q)\,dq.
\end{equation}
Then
\begin{equation}\label{eq:phi-properties}
 \phi_\varepsilon(r)=\phi_\varepsilon'(r)=0\quad(r\le0),
 \qquad
 0\le\phi_\varepsilon''\le1,
 \qquad
 0\le\phi_\varepsilon'(r)\le r^+,
 \qquad
 0\le\phi_\varepsilon(r)\le\frac12(r^+)^2,
\end{equation}
and, as $\varepsilon\downarrow0$,
\begin{equation}\label{eq:phi-limits}
 \phi_\varepsilon(r)\to\frac12(r^+)^2,
 \qquad
 \phi_\varepsilon'(r)\to r^+,
 \qquad
 \phi_\varepsilon''(r)\to\1_{\{r>0\}}.
\end{equation}

\begin{proposition}[Positive-part It\^o identity]\label{prop:positive-ito}
Fix $N$ and write $Y=Y^{N,+}$ and $K=K^{N,+}$. Let
\[
 w(t,x):=u_N(t,x)-Y_t.
\]
Then, for every $t\ge0$, almost surely,
\begin{align}
 \frac12\norm{w^+(t)}_{L^2(D)}^2
 &+\frac12\int_0^t\norm{\nabla w^+(s)}_{L^2(D)}^2\,ds\notag\\
 &=\frac12\norm{w^+(0)}_{L^2(D)}^2
 +\int_0^t\int_D w^+(s,x)
 \bigl[b_N(u_N(s,x))-b_N(Y_s)\bigr]dx\,ds\notag\\
 &\quad+\frac12\int_0^t\int_D
 \1_{\{w(s,x)>0\}}
 |\sigma_N(u_N(s,x))-\sigma_N(Y_s)|^2\,dx\,ds\notag\\
 &\quad+\int_0^t
 \left[\int_D w^+(s,x)
 \bigl(\sigma_N(u_N(s,x))-\sigma_N(Y_s)\bigr)dx\right]dB_s\notag\\
 &\quad-\int_0^t\left[\int_D w^+(s,x)\,dx\right]dK_s.
 \label{eq:positive-ito}
\end{align}
For the lower comparison, the analogous identity holds with \(w=Y^{N,-}-u_N\), \(K=K^{N,-}\), and the coefficient differences replaced by \(b_N(Y^{N,-})-b_N(u_N)\) and \(\sigma_N(Y^{N,-})-\sigma_N(u_N)\), respectively.
\end{proposition}

\begin{proof}
We give the upper-barrier argument. Since $Y_t\ge0$ and $u_N$ has zero boundary trace,
\[
 \operatorname{Tr}w(t)=-Y_t\le0.
\]
Hence $\phi_\varepsilon'(w(t))\in H_0^1(D)$ for almost every $t$. The Sobolev chain rule gives
\begin{equation}\label{eq:chain}
 \nabla\phi_\varepsilon'(w)
 =\phi_\varepsilon''(w)\nabla u_N.
\end{equation}
Set
\[
 f_N:=b_N(u_N)-b_N(Y),
 \qquad g_N:=\sigma_N(u_N)-\sigma_N(Y).
\]
The difference satisfies, in the weak sense,
\[
 dw=\frac12\Delta u_N\,dt+f_N\,dt+g_N\,dB_t-dK_t.
\]

To get It\^o's formula, set
$w_\delta=S(\delta)u_N-Y$, where $S$ is the Dirichlet
heat semigroup. Apply the one-dimensional It\^o formula
to $\phi_\varepsilon(w_\delta(t,x))$ and integrate over $D$.
Integration by parts produces no boundary term because
$\phi_\varepsilon'(-Y)=0$.
After localization, let $\delta\downarrow0$.
The reflection term converges by the Lipschitz continuity
of $\phi_\varepsilon'$ and the boundedness of $K$
on every bounded time interval. We obtain
\begin{align}
\int_D\phi_\varepsilon(w(t,x))\,dx
&=\int_D\phi_\varepsilon(w(0,x))\,dx
-\frac12\int_0^t\int_D
\phi_\varepsilon''(w)|\nabla u_N|^2\,dx\,ds\notag\\
&\quad+\int_0^t\int_D\phi_\varepsilon'(w)f_N\,dx\,ds
+\frac12\int_0^t\int_D\phi_\varepsilon''(w)|g_N|^2\,dx\,ds\notag\\
&\quad+\int_0^t
\left[\int_D\phi_\varepsilon'(w)g_N\,dx\right]dB_s
-\int_0^t\left[\int_D\phi_\varepsilon'(w)\,dx\right]dK_s.
\label{eq:ito-eps}
\end{align}
By \eqref{eq:chain},
\begin{align}\label{eq:laplace-eps}
	\ip{\Delta u_N}{\phi_\varepsilon'(w)}_{H^{-1},H_0^1}
	&=-\int_D\nabla u_N\cdot\nabla\phi_\varepsilon'(w)\,dx\\
	&=-\int_D\phi_\varepsilon''(w)|\nabla u_N|^2\,dx.\notag
\end{align}

Now let $\varepsilon\downarrow0$. The drift and diffusion terms converge by dominated convergence, using the global Lipschitz bounds and \eqref{eq:phi-properties}. For the gradient term,
\[
 \int_0^t\int_D\phi_\varepsilon''(w)|\nabla u_N|^2\,dx\,ds
 \longrightarrow
 \int_0^t\int_D\1_{\{w>0\}}|\nabla u_N|^2\,dx\,ds
 =\int_0^t\norm{\nabla w^+(s)}_2^2\,ds.
\]
The stochastic integrals converge after localization. Finally,
$0\le r^+-\phi_\varepsilon'(r)\le\varepsilon$, so the error in the reflection term is at most $\varepsilon|D|K_T$. This proves \eqref{eq:positive-ito}.

For the lower comparison, use $w=Y^{N,-}-u_N$. Its boundary trace is $Y^{N,-}\le0$, and its finite-variation reflection term is again $-dK^{N,-}$. The same argument therefore applies.
\end{proof}

\subsection{Comparison for the truncated equations}

\begin{lemma}[Stochastic comparison]\label{lem:comparison}
For every $N>M+1$, there is an event of probability one on which
\begin{equation}\label{eq:truncated-comparison}
 Y_t^{N,-}\le u_N(t,x)\le Y_t^{N,+}
\end{equation}
for every $t\ge0$ and almost every $x\in D$.
\end{lemma}

\begin{proof}
	The following comparison argument is related to comparison principles
	for function-valued stochastic equations; see
	\cite{DenisMatoussiStoica2009,DenisMatoussi2013,Kotelenez1992}.
	Our proof is adapted to the present setting, where the noise is driven
	by a single Brownian motion and the comparison processes are reflected scalar
	SDE barriers.
	
We prove the upper bound. Set $Y=Y^{N,+}$ and $w=u_N-Y$. Initially,
\[
 w(0,x)=u_0(x)-M\le0,
\]
so $w^+(0)=0$.

Let $L_{b,N}$ and $L_{\sigma,N}$ be global Lipschitz constants of $b_N$ and $\sigma_N$. On $\{w>0\}$,
\begin{align}
 w^+\bigl[b_N(u_N)-b_N(Y)\bigr]
 &\le L_{b,N}(w^+)^2,
 \label{eq:b-lip-positive}\\
 |\sigma_N(u_N)-\sigma_N(Y)|^2
 &\le L_{\sigma,N}^2(w^+)^2.
 \label{eq:sigma-lip-positive}
\end{align}
Applying Proposition \ref{prop:positive-ito}, dropping the nonnegative gradient term on the left and the nonpositive reflection term on the right, localizing the martingale, and taking expectations gives
\begin{equation}\label{eq:comparison-pregronwall}
 \frac12\Ee\norm{w^+(t\wedge\kappa)}_2^2
 \le
 \left(L_{b,N}+\frac12L_{\sigma,N}^2\right)
 \int_0^t
 \Ee\!\left[\1_{\{s\le\kappa\}}\norm{w^+(s)}_2^2\right]ds
\end{equation}
for every bounded localization time $\kappa$.

For fixed $T$ and $m\ge1$, take
\[
 \kappa_m:=T\wedge\inf\{t\ge0:\norm{w^+(t)}_2^2\ge m\}.
\]
The $L^2$-continuity of $u_N$ and continuity of $Y$ make $\kappa_m$ a stopping time. Before $\kappa_m$, the martingale integrand in \eqref{eq:positive-ito} is square integrable since
\[
 \left|\int_Dw^+\bigl(\sigma_N(u_N)-\sigma_N(Y)\bigr)dx\right|
 \le L_{\sigma,N}\norm{w^+}_2^2.
\]
Set
\[
 F_m(t):=\Ee\norm{w^+(t\wedge\kappa_m)}_2^2.
\]
Then \eqref{eq:comparison-pregronwall} and Gronwall's lemma give $F_m(t)=0$ on $[0,T]$. Hence the level $m$ is never reached and
\[
 u_N(t,x)\le Y_t^{N,+}
\]
for every $t\in[0,T]$ and almost every $x$, on one event of probability one. The passage from fixed times to all times uses a countable dense set and $L^2$-continuity.

For the lower bound, set $w=Y^{N,-}-u_N$. Then $w(0)\le0$, its boundary trace is $Y^{N,-}\le0$, and on $\{w>0\}$ the same Lipschitz estimates hold with $Y^{N,-}$ and $u_N$ interchanged. The reflection term is again nonpositive. Thus $w^+=0$. Taking a countable intersection over integer time horizons proves \eqref{eq:truncated-comparison} simultaneously for all $t\ge0$.
\end{proof}

\section{Global solutions}
Let $(X^\pm,K^\pm)$ be the global reflected solutions of \eqref{eq:scalar-barriers}, and define
\begin{equation}\label{eq:rhoN}
 \rho_N:=\inf\left\{t\ge0:
 |X_t^+|\vee|X_t^-|\ge N\right\}.
\end{equation}
By continuity and global non-explosion of the reflected scalar solutions,
\begin{equation}\label{eq:rhoN-infty}
 \rho_N\uparrow\infty
 \qquad\text{almost surely.}
\end{equation}
Because $b_N=b$ and $\sigma_N(r)=\sigma(r)$ on $[-N,N]$, pathwise uniqueness for the reflected scalar equations gives
\begin{equation}\label{eq:Y-X}
 Y_t^{N,\pm}=X_t^\pm,
 \qquad K_t^{N,\pm}=K_t^\pm,
 \qquad 0\le t<\rho_N.
\end{equation}
Combining \eqref{eq:Y-X} with Lemma \ref{lem:comparison} gives
\begin{equation}\label{eq:uN-barrier}
 X_t^-\le u_N(t,x)\le X_t^+,
 \qquad 0\le t<\rho_N,
\end{equation}
for almost every $x$. Since \eqref{eq:scalar-sign} holds and $|X_t^\pm|<N$ for $t<\rho_N$,
\begin{equation}\label{eq:uN-under-N}
 |u_N(t,x)|<N,
 \qquad 0\le t<\rho_N.
\end{equation}
Hence both truncations are inactive:
\[
 b_N(u_N)=b(u_N),
 \qquad
 \sigma_N(u_N)=\sigma(u_N),
 \qquad t<\rho_N.
\]
Thus $u_N$ solves the original equation \eqref{eq:spde} up to $\rho_N$.

\begin{lemma}[Consistency of truncations]\label{lem:consistency}
If $K>N>M+1$, then
\begin{equation}\label{eq:consistency}
 u_K(t)=u_N(t),
 \qquad 0\le t<\rho_N,
\end{equation}
almost surely in $L^2(D)$.
\end{lemma}

\begin{proof}
Since $\rho_N\le\rho_K$, pathwise uniqueness for the reflected scalar equations gives
\[
 Y_t^{K,\pm}=X_t^\pm,
 \qquad t<\rho_N.
\]
The comparison lemma applied to $u_K$ then gives
\[
 X_t^-\le u_K(t,x)\le X_t^+,
 \qquad t<\rho_N,
\]
so $|u_K|<N$ on this interval. Consequently,
\[
 b_K(u_K)=b(u_K)=b_N(u_K),
 \qquad
 \sigma_K(u_K)=\sigma(u_K)=\sigma_N(u_K),
 \qquad t<\rho_N.
\]
The same identities hold for $u_N$. Hence, up to $\rho_N$, both processes solve the same globally Lipschitz truncated SPDE with coefficients $b_N,\sigma_N$, the same Brownian motion, and the same initial condition. Pathwise uniqueness for that equation yields \eqref{eq:consistency}.
\end{proof}


\begin{proof}[Proof of Theorem \ref{thm:main}]
\textbf{Step 1: construction by pasting.}
Work on the probability-one event on which \eqref{eq:rhoN-infty} and Lemma \ref{lem:consistency} hold for all integers $N>M+1$. For $t\ge0$, choose any $N$ such that $t<\rho_N$ and define
\begin{equation}\label{eq:pasting}
 u(t):=u_N(t).
\end{equation}
Such an $N$ exists by \eqref{eq:rhoN-infty}, and the definition is independent of the choice of $N$ by consistency.

An explicit adapted version can be obtained by fixing an increasing sequence $N_j\to\infty$ and writing, for each fixed $t$,
\[
 u(t)=\mathbf1_{\{t<\rho_{N_1}\}}u_{N_1}(t)+\sum_{j\ge2}
 \1_{\{\rho_{N_{j-1}}\le t<\rho_{N_j}\}}u_{N_j}(t).
\]
 Since $\rho_N$ are stopping times, the indicators are $\mathcal F_t$-measurable.

For each fixed $T<\infty$ and almost every sample path, choose $N=N(\omega,T)$ so large that $T<\rho_N(\omega)$. Then
\[
 u(t,\omega)=u_N(t,\omega),
 \qquad 0\le t\le T.
\]
Hence $u$ inherits
\[
 u\in C([0,T];L^2(D))\cap L^2(0,T;H_0^1(D)).
\]
On $[0,\rho_N)$ the truncated coefficients equal the original coefficients along $u_N$, so the stopped variational identities paste together to give \eqref{eq:variational-form} on every finite interval.

\medskip
\textbf{Step 2: the barrier bound and non-explosion.}
For $t<\rho_N$, \eqref{eq:uN-barrier} gives
\begin{equation}\label{4.0}
 X_t^-\le u(t,x)\le X_t^+
\end{equation}
for almost every $x$. Since $\rho_N\uparrow\infty$, the above inequality holds for all finite times. For every $T<\infty$,
\[
 C_T(\omega):=
 \max\left\{
 \sup_{0\le s\le T}X_s^+(\omega),
 \sup_{0\le s\le T}(-X_s^-(\omega))
 \right\}<\infty
\]
because $X^\pm$ have continuous global paths. Thus
\[
 \sup_{0\le t\le T}\norm{u(t)}_{L^\infty(D)}\le C_T
\]
almost surely, proving the pathwise non-explosion of the solution.

\medskip
\textbf{Step 3: uniqueness in the stated solution class.}
Let $u$ and $v$ be two global solutions in the class \eqref{eq:solution-class}--\eqref{eq:variational-form}, driven by the same Brownian motion and with the same initial condition. Define
\begin{equation}\label{eq:uniq-stop}
 \tau_R:=\inf\left\{t\ge0:
 \norm{u(t)}_{L^\infty(D)}\vee
 \norm{v(t)}_{L^\infty(D)}\ge R\right\}.
\end{equation}
By progressive measurability, $\tau_R$ is a stopping time, and local $L^\infty$ boundedness implies
\begin{equation}\label{eq:tauR-infty}
 \tau_R\uparrow\infty
 \qquad\text{a.s.}
\end{equation}
For $R>M+1$, take $\widehat b_R=b\circ\pi_R$ and $\widehat\sigma_R=\sigma\circ\pi_R$, which are bounded globally Lipschitz functions agreeing with $b$ and $\sigma$ on $[-R,R]$. On $[0,\tau_R)$ both $u$ and $v$ solve the same globally Lipschitz truncated SPDE
\[
 dw=\frac12\Delta w\,dt+\widehat b_R(w)\,dt+\widehat\sigma_R(w)\,dB_t.
\]
Pathwise uniqueness for that equation yields
\[
 u(t)=v(t),
 \qquad 0\le t<\tau_R.
\]
Letting $R\to\infty$ and using \eqref{eq:tauR-infty} gives $u=v$ for all $t\ge0$ almost surely.
\end{proof}

\begin{remark}
Moreover, using the heat kernel estimates of the Laplacian operator one can show that  $u$ has a spatially continuous version. Consequently, the inequalities in \eqref{4.0} extend from almost every \(x\in D\) to every \(x\in D\).
\end{remark}

\vskip 0.3cm
\noindent{\large\bf Acknowledgements}

\begingroup
\emergencystretch=2em
We thank Fengyu Wang and Michael Salins for very useful comments and suggestions. This work is partially supported by the National Key R\&D Program of China (No. 2022YFA1006001) and the National Natural Science Foundation of China (Nos. 12131019, 12571158, 12371151). 
\par\endgroup

\end{document}